\newif\ifarxiv
\arxivtrue
\ifarxiv
\else
\fi

\documentclass[a4paper,leqno]{article}

\usepackage[utf8]{inputenc} \usepackage[usenames,dvipsnames]{xcolor}
\usepackage{amsmath}					\numberwithin{equation}{section}

\usepackage{hyphenat}

\DeclareSymbolFont{bbold}{U}{bbold}{m}{n}
\DeclareSymbolFontAlphabet{\mathbbold}{bbold}

\usepackage{amssymb}
\usepackage{euscript}
\usepackage{mathrsfs}
\usepackage{bbm}
\usepackage[amsmath,thmmarks]{ntheorem}	\usepackage[ntheorem]{empheq}
\usepackage[hmarginratio={1:1},textwidth=400pt]{geometry}
\usepackage[normalem]{ulem}

\usepackage{tikz}
\usetikzlibrary{cd}

\usepackage[all]{xy}

\usepackage[colorlinks=true,linkcolor=blue,anchorcolor=blue,citecolor=blue,
     		filecolor=blue,urlcolor=blue]{hyperref}
\usepackage{stmaryrd}
\SetSymbolFont{stmry}{bold}{U}{stmry}{m}{n} 

\usepackage{bussproofs}
\EnableBpAbbreviations

\usepackage{mathtools}
\mathtoolsset{
showonlyrefs=true }

\usepackage{array}
\usepackage{tabularx}
\usepackage{enumitem}

\usepackage[shortcuts]{extdash}

\usepackage{xspace}
\DeclareMathAlphabet{\mathpzc}{OT1}{pzc}{m}{it}
\xspaceremoveexception{-}
\makeatletter
\renewcommand*\@xspace@hook{\ifx\@let@token-\expandafter\@xspace@dash@i
\fi
}
\def\@xspace@dash@i-{\futurelet\@let@token\@xspace@dash@ii}
\def\@xspace@dash@ii{\ifx\@let@token-\else
\unskip
\fi
-}
\makeatother

\usepackage[small]{titlesec}
\setlist[enumerate]{label=(\arabic*)}

\newcommand{\PAsm}{\mathsf{PAsm}}
\newcommand{\downcl}{\mathopen\downarrow}

\newcommand{\nin}{{n\in\N}}
\newcommand{\hs}{_\#}

\newcommand{\angs}[1]{\langle #1\rangle}

\newcommand{\sto}{\,\to\,}

\newcommand{\pit}{\pi_2}

\newcommand{\mcc}{\mathcal{C}}
\newcommand{\mcf}{\mathcal{F}}
\newcommand{\mfa}{\mathfrak{A}}

\newcommand{\Ord}{\mathsf{Ord}}
\newcommand{\IOrd}{\mathsf{IOrd}}

\newcommand{\hypk}{{\EuScript{K}}}

\newcommand{\trip}{{\EuScript{P}}}

\newcommand{\comk}{\mathsf{k}}				\newcommand{\coms}{\mathsf{s}}

\newcommand{\msep}{\mathrel|}		\newcommand{\op}{^\mathsf{op}}						\newcommand{\id}{\mathrm{id}}				

\newcommand{\ap}{\mathclose\cdot}			\newcommand{\qdot}{\,.\,}				

\newcommand{\setof}[2]{\{#1\msep #2\}}
\newcommand{\fami}[2]{(#1)_{#2}}

\newcommand{\tttc}{tripos-to-topos construction\xspace}

\newcommand{\N}{\mathbb{N}}
\newcommand{\R}{\mathbb{R}}

\newcommand{\xto}[1]{\stackrel{#1}{\to}}
\newcommand{\incl}{\hookrightarrow}
\newcommand{\epi}{\twoheadrightarrow}
\newcommand{\epil}{\twoheadleftarrow}

\newcommand{\vf}{\varphi}

\newcommand{\imp}{\Rightarrow}

\newcommand{\famf}{\mathsf{fam}}

\newcommand{\adj}{\dashv}

\newcommand{\pto}{\rightharpoonup}

\newcommand{\DCO}{\mathsf{DCO}}
\newcommand{\BCO}{\mathsf{BCO}}

\newcommand{\brar}{{(A,R)}}
\newcommand{\brbs}{{(B,S)}}

\newcommand{\UOrd}{\mathsf{UOrd}}

\newcommand{\fifa}{\EuScript{A}}
\newcommand{\fifh}{\EuScript{H}}
\newcommand{\fifk}{\EuScript{K}}
\newcommand{\fifp}{\EuScript{P}}

\newcommand{\down}{{\downarrow}}

\newcommand{\prim}{\mathsf{prim}}

\def\signed #1{{\leavevmode\unskip\nobreak\hfil\penalty50\hskip2em
\hbox{}\nobreak\hfil(#1)\parfillskip=0pt \finalhyphendemerits=0 \endgraf}}

\newsavebox\mybox

\newcommand{\gcons}{{\textstyle\int}}

\newcommand{\setord}{\Set\op\to\Ord}
\newcommand{\image}{\mathsf{im}}
\newcommand{\subs}{\subseteq}

\newcommand{\fa}{\forall}
\newcommand{\ex}{\exists}
\newcommand{\sfa}{\;\forall}
\newcommand{\sex}{\;\exists}

\makeatletter
\DeclareRobustCommand\onedot{\futurelet\@let@token\@onedot}
\def\@onedot{\ifx\@let@token.\else.\null\fi\xspace}
\makeatother

\newcommand{\ie}{i.e\onedot}

\newcommand{\eprime}{$\exists$-prime\xspace}
\newcommand{\ecompletion}{$\exists$-completion\xspace}
\let\lim\relax\DeclareMathOperator*{\lim}{\mathsf{lim}}

\DeclareFontFamily{U}{min}{}
\DeclareFontShape{U}{min}{m}{n}{<-> udmj30}{}
\newcommand\yo{\!\text{\usefont{U}{min}{m}{n}\symbol{'210}}}

\newcommand{\qtext}[1]{\quad\text{#1}\quad}
\newcommand{\qqtext}[1]{\qquad\text{#1}\qquad}

\newcommand{\qqdefequi}{\qquad:\Leftrightarrow\qquad}

\newenvironment{inumerate}{\begin{enumerate}[label=\normalfont(\roman*)]}{\end{enumerate}}

\newcommand{\h}{\text{-}}
\newcommand{\Set}{\mathsf{Set}}

\newcommand{\CC}{\mathbb{C}}

\makeatletter
\newcommand*{\twoheadrightsquigarrow}{\rightsquigarrow\joinrel\mathrel{\mathpalette\@twoheadrightsquigarrow\relax}}
\newcommand*{\@twoheadrightsquigarrow}[2]{\clipbox{{.7\width} 0pt 0pt {-.2\height}}{$\m@th#1\rightsquigarrow$}}
\makeatother

\DeclareFontFamily{U} {MnSymbolC}{}
\DeclareFontShape{U}{MnSymbolC}{m}{n}{
  <-6> MnSymbolC5
  <6-7> MnSymbolC6
  <7-8> MnSymbolC7
  <8-9> MnSymbolC8
  <9-10> MnSymbolC9
  <10-12> MnSymbolC10
  <12-> MnSymbolC12}{}
\DeclareFontShape{U}{MnSymbolC}{b}{n}{
  <-6> MnSymbolC-Bold5
  <6-7> MnSymbolC-Bold6
  <7-8> MnSymbolC-Bold7
  <8-9> MnSymbolC-Bold8
  <9-10> MnSymbolC-Bold9
  <10-12> MnSymbolC-Bold10
  <12-> MnSymbolC-Bold12}{}
\DeclareSymbolFont{MnSyC} {U} {MnSymbolC}{m}{n}
\DeclareMathSymbol{\smalltriangleright}{\mathbin}{MnSyC}{72}

\newcommand{\x}{\times}

\newcommand{\ox}{\otimes}

 
 \makeatletter
\providecommand*{\twoheadrightarrowfill@}{\arrowfill@\relbar\relbar\twoheadrightarrow
}
\providecommand*{\twoheadleftarrowfill@}{\arrowfill@\twoheadleftarrow\relbar\relbar
}
\providecommand*{\xepi}[2][]{\ext@arrow 0579\twoheadrightarrowfill@{#1}{#2}}
\providecommand*{\xepil}[2][]{\ext@arrow 5097\twoheadleftarrowfill@{#1}{#2}}
\makeatother

\makeatletter
\newbox\xrat@below
\newbox\xrat@above
\newcommand{\xmono}[2][]{\setbox\xrat@below=\hbox{\ensuremath{\scriptstyle #1}}\setbox\xrat@above=\hbox{\ensuremath{\scriptstyle #2}}\pgfmathsetlengthmacro{\xrat@len}{max(\wd\xrat@below,\wd\xrat@above)+.6em}\mathrel{\tikz [>->,baseline=-.75ex]
                 \draw (0,0) -- node[below=-2pt] {\box\xrat@below}
                                node[above=-2pt] {\box\xrat@above}
                       (\xrat@len,0) ;}}
\makeatother

\makeatletter
\newcommand*{\relrelbarsep}{.386ex}
\newcommand*{\relrelbar}{\mathrel{\mathpalette\@relrelbar\relrelbarsep
  }}
\newcommand*{\@relrelbar}[2]{\raise#2\hbox to 0pt{$\m@th#1\relbar$\hss}\lower#2\hbox{$\m@th#1\relbar$}}
\providecommand*{\rightrightarrowsfill@}{\arrowfill@\relrelbar\relrelbar\rightrightarrows
}
\providecommand*{\leftleftarrowsfill@}{\arrowfill@\leftleftarrows\relrelbar\relrelbar
}
\providecommand*{\xrightrightarrows}[2][]{\ext@arrow 0359\rightrightarrowsfill@{#1}{#2}}
\providecommand*{\xleftleftarrows}[2][]{\ext@arrow 3095\leftleftarrowsfill@{#1}{#2}}
\makeatother

\makeatletter
\newbox\xrat@below
\newbox\xrat@above
\newcommand{\xrightarrowtail}[2][]{\setbox\xrat@below=\hbox{\ensuremath{\scriptstyle #1}}\setbox\xrat@above=\hbox{\ensuremath{\scriptstyle #2}}\pgfmathsetlengthmacro{\xrat@len}{max(\wd\xrat@below,\wd\xrat@above)+.6em}\mathrel{\tikz [>->,baseline=-.75ex]
                 \draw (0,0) -- node[below=-2pt] {\box\xrat@below}
                                node[above=-2pt] {\box\xrat@above}
                       (\xrat@len,0) ;}}
\makeatother
 \newtheorem{theorem}{Theorem}[section]
  \AtBeginEnvironment{theorem}{\setlist[enumerate]{label=\normalfont(\roman*)}}
\newtheorem{proposition}[theorem]{Proposition}
  \AtBeginEnvironment{proposition}
        {\setlist[enumerate]{label=\normalfont(\roman*)}}
\newtheorem{lemma}[theorem]{Lemma}
  \AtBeginEnvironment{lemma}{\setlist[enumerate]{label=\normalfont(\roman*)}}

  \AtBeginEnvironment{conjecture}
        {\setlist[enumerate]{label=\normalfont(\roman*)}}
\newtheorem{corollary}[theorem]{Corollary}
  \AtBeginEnvironment{corollary}
        {\setlist[enumerate]{label=\normalfont(\roman*)}}

  \AtBeginEnvironment{claim}{\setlist[enumerate]{label=\normalfont(\roman*)}}
\theorembodyfont{\normalfont}
\theoremsymbol{\ensuremath{\diamondsuit}}

  \AtBeginEnvironment{exercise}{\setlist[enumerate]{label=(\alph*)}}
\newtheorem{remark}[theorem]{Remark}
  \AtBeginEnvironment{remark}{\setlist[enumerate]{label=(\alph*)}}

  \AtBeginEnvironment{remnot}{\setlist[enumerate]{label=(\alph*)}}

  \AtBeginEnvironment{notation}{\setlist[enumerate]{label=(\alph*)}}

  \AtBeginEnvironment{terminology}
        {\setlist[enumerate]{label=\normalfont(\alph*)}}
\newtheorem{remarks}[theorem]{Remarks}
  \AtBeginEnvironment{remarks}{\setlist[enumerate]{label=(\alph*)}}

  \AtBeginEnvironment{example}{\setlist[enumerate]{label=(\alph*)}}
\newtheorem{examples}[theorem]{Examples}
  \AtBeginEnvironment{examples}{\setlist[enumerate]{label=(\alph*)}}

  \AtBeginEnvironment{thought}{\setlist[enumerate]{label=(\alph*)}}

  \AtBeginEnvironment{convention}{\setlist[enumerate]{label=(\alph*)}}

  \AtBeginEnvironment{conventions}{\setlist[enumerate]{label=(\alph*)}}
\newtheorem{void}[theorem]{$\!\!$}
  \AtBeginEnvironment{void}{\setlist[enumerate]{label=(\alph*)}}

\newtheorem{definition}[theorem]{Definition}
  \AtBeginEnvironment{definition}{\setlist[enumerate]{label=(\roman*)}}
\theoremsymbol{\ensuremath{\square}}
\qedsymbol{\ensuremath{\square}}
\theoremstyle{nonumberplain}
\theoremheaderfont{\itshape}
\newtheorem{proof}{Proof.}

\theorembodyfont{\itshape}
\theoremheaderfont{\normalfont\bfseries}

\renewcommand*\descriptionlabel[1]{\hspace\labelsep\normalfont #1}

\makeatletter
\let\orgdescriptionlabel\descriptionlabel
\renewcommand*{\descriptionlabel}[1]{\let\orglabel\label
  \let\label\@gobble
  \phantomsection
  \edef\@currentlabel{#1}\let\label\orglabel
  \orgdescriptionlabel{#1}}
\makeatother

\setbox0=\hbox{$A$}\newdimen\axis \axis=\fontdimen22\textfont2\newbox\bullbox \setbox\bullbox=\hbox{.}
 \wd\bullbox 0pt
 \ht\bullbox 2.0pt
 \dp\bullbox 1.0pt
\newcommand{\dashed}{\ar@{-->}}

 \def\bull{\kern 1.9pt \usebox\bullbox\kern 1.1pt}
\newdir{..}{\object{\bull}}
\newcommand{\dotted}{\ar@{..>}}
\newcommand{\dline}{\ar@{..}}
\newcommand{\dar}{\ar@{..>}}
\newcommand{\dreg}{\ar@{|..>}}
\newcommand{\ddmono}{\ar@{ >-->}}
\newcommand{\dcocov}{\ar@{ |>..>}}
\newcommand{\ddepi}{\ar@{..>>}}
\newcommand{\dcov}{\ar@{..|>}}
\newcommand{\dmonepi}{\ar@{ >..>>}}

\newcommand{\midar}{\ar@{-}|-(.63)*=0@{>}}
\newcommand{\dmidar}{\ar@{..}|-(.63)*=0@{>}}
\newcommand{\midarr}{\ar@{-}|-(.7)*=0@{>>}}
\newcommand{\dmidarr}{\ar@{..}|-(.7)*=0@{>>}}

\newcommand{\dashprof}{\ar@{-->}|-*=0@{|}}
\newdir{+>}{{}*!/-4.5pt/@{>}}
\newcommand{\mar}{\ar@{-}|-*=0@{+>}}
\newcommand{\nohead}{\ar@{-}}
\newcommand{\dnohead}{\ar@{..}}
\newcommand{\emar}{\ar@{}}

\newcommand{\cart}{\ar@{~>}}
\newcommand{\depicart}{\ar@{~|>}}
\newdir{||}{{}*!/-5pt/@{}}
\newdir{ >|}{{}*!/-8pt/@{>}*!/-16pt/@{}}
\newcommand{\coca}{\ar@{||{+} >|}}
\newcommand{\arid}{\ar@{=}}

\newcommand{\uord}{uniform preorder\xspace}
\newcommand{\uords}{uniform preorders\xspace}
\newcommand{\iord}{indexed preorder\xspace}

\newcommand{\tfae}{The following are equivalent\xspace}

\newcommand{\loc}{locally ordered category\xspace}

\newcommand{\hex}{\ex\h}
\makeatletter
\DeclareFontFamily{OMX}{MnSymbolE}{}
\DeclareSymbolFont{MnLargeSymbols}{OMX}{MnSymbolE}{m}{n}
\SetSymbolFont{MnLargeSymbols}{bold}{OMX}{MnSymbolE}{b}{n}
\DeclareFontShape{OMX}{MnSymbolE}{m}{n}{
    <-6>  MnSymbolE5
   <6-7>  MnSymbolE6
   <7-8>  MnSymbolE7
   <8-9>  MnSymbolE8
   <9-10> MnSymbolE9
  <10-12> MnSymbolE10
  <12->   MnSymbolE12
}{}
\DeclareFontShape{OMX}{MnSymbolE}{b}{n}{
    <-6>  MnSymbolE-Bold5
   <6-7>  MnSymbolE-Bold6
   <7-8>  MnSymbolE-Bold7
   <8-9>  MnSymbolE-Bold8
   <9-10> MnSymbolE-Bold9
  <10-12> MnSymbolE-Bold10
  <12->   MnSymbolE-Bold12
}{}
\let\llangle\@undefined
\let\rrangle\@undefined
\DeclareMathDelimiter{\llangle}{\mathopen}{MnLargeSymbols}{'164}{MnLargeSymbols}{'164}
\DeclareMathDelimiter{\rrangle}{\mathclose}{MnLargeSymbols}{'171}{MnLargeSymbols}{'171}
\makeatother

\newcommand{\aahd}{(A,\cdot,A\hs)}

\newcommand{\dfahs}{\down\mcf_{A\hs}}
\newcommand{\nwedge}[1]{\wedge^{(#1)}}
\newcommand{\npi}[2]{\pi^{(#1)}_{#2}}
\newcommand{\ccomp}{relationally complete\xspace}
\newcommand{\ccompness}{relational completeness\xspace}

\newcommand{\Ccompness}{Relational completeness\xspace}

\begin{document} 
\title{Uniform Preorders \\ and Partial Combinatory Algebras}
\author{Jonas Frey\thanks{This work is partially supported by the Air Force
Office of Scientific Research under award number FA9550-20-1-0305, and by the
Army Research Office under award number W911NF-21-1-0121.}}
\maketitle
\begin{abstract}
\emph{Uniform preorders} are a class of combinatory representations of
$\Set$-indexed preorders that generalize Hofstra's \emph{basic relational
objects}~\cite{hofstra2006all}. An indexed preorder is representable by a
uniform preorder if and only if it has as generic predicate. We study the
$\ex$-completion of indexed preorders on the level of uniform preorders, and
identify a combinatory condition (called `relational completeness') which
characterizes those uniform preorders with finite meets whose $\ex$-completions
are triposes. The class of triposes obtained this way contains \emph{relative
realizability triposes}, for which we derive a characterization as a fibrational
analogue of the characterization of realizability toposes given in earlier
work~\cite{frey2019characterizing}.

Besides relative partial combinatory algebras, the class of relationally
complete uniform preorders contains \emph{filtered ordered} partial combinatory
algebras, and it is unclear if there are any others.
\end{abstract}

\tableofcontents

\section*{Introduction}\addcontentsline{toc}{section}{Introduction}

In his seminal article~\cite{hofstra2006all}, Pieter Hofstra gave an analysis of
\emph{filtered ordered combinatory algebras} (filtered OPCAs) in terms of the
more primitive notion of \emph{basic combinatory objects} (BCOs). These are
combinatory representations $(A,\leq,\mcf)$ of certain $\Set$-indexed preorders
by partial orders equipped with a class of partial endomaps, and Hofstra showed
that a BCO $(A,\leq,\mcf)$ arises from {a filtered OPCA}
if and only if
\begin{enumerate}
\item it is \emph{cartesian} in the sense that the associated indexed preorder
$\famf(A,\leq,\mcf)$ is an indexed meet-semilattice, and 
\item the free completion under existential quantification (`$\ex$-completion')
of $\famf(A,\leq,\mcf)$ is a tripos.
\end{enumerate}
The present work gives two variations on this theme, replacing BCOs by the more
general notion of \emph{uniform preorder} on the one hand, and by the more
restrictive notion of \emph{discrete combinatory object} on the other hand,
together fitting into a sequence
\begin{equation}
\DCO\to\BCO\to\UOrd\to\IOrd
\end{equation}
of embeddings of \emph{locally ordered categories}. A uniform preorder is a set
equipped with a monoid of \emph{binary relations} (Definition~\ref{def:uord}),
and a DCO is a set with a monoid of \emph{partial functions}
(Definition~\ref{def:discrete}\ref{def:discrete:dco}), and the locally ordered
categories $\DCO$ and $\UOrd$ have the advantage over $\BCO$ that their
bi-essential images in the locally ordered category $\IOrd$ of $\Set$-indexed
preorders admit straightforward characterizations: an indexed preorder is
representable by an uniform preorder iff it has a generic predicate
(Lemma~\ref{lem:genpred-image}), and it is representable by a DCO iff it has a
\emph{discrete} generic predicate (Corollary~\ref{cor:dco-discgen}).

After developing the basic theory of uniform preorders in
Sections~\ref{se:locat-uord}--\ref{se:ifrms}, we give a combinatorial criterion
for the $\ex$-completion of a cartesian uniform preorder to be a tripos in
Definition~\ref{def:rcomp} and Theorem~\ref{thm:rel-compl}, which we call
\emph{relational completeness}. In
Example~\ref{ex:comcomp}\ref{ex:comcomp:tripos}, relational completeness is used
to show that the $\ex$-completion of a tripos is again a tripos, and
Remark~\ref{rem:relcompl}\ref{rem:relcompl:tripos-char} gives a characterization
of the triposes that arise as $\ex$-completions of (the indexed preorders
associated to) relationally complete uniform preorders, building on a prior
characterization of $\ex$-completions in terms of \emph{$\exists$-prime
predicates} (Proposition~\ref{prop:ecomp-if-enough-primes}). This
characterization is augmented by a discreteness condition in
Theorem~\ref{thm:rtr-char} to obtain a characterization of relative
realizability triposes: 
\begin{quote}
\emph{A tripos $\trip$ is a relative realizability tripos if and only if it has
enough $\exists$-prime predicates, and the indexed sub-preorder $\prim(\trip)$
of prime predicates has finite meets and a discrete generic predicate.}
\end{quote}
In light of the close analogy between Theorem~\ref{thm:rtr-char} and
Remark~\ref{rem:relcompl}\ref{rem:relcompl:tripos-char}, relationally complete
uniform preorders could be viewed as (relative/filtered) \emph{relational PCAs}.

A central question remains open: every filtered OPCA gives rise to a
relationally complete uniform preorder, but are there any others?

\medskip

Most of the work presented here is already contained in the author's PhD
thesis~\cite{frey2013fibrational}, where the theory of uniform preorders is
developed in greater generality, including \emph{many-sorted} uniform preorders,
and without the use of the axiom of choice. To get a more accessible
presentation, we have left out the subtleties of a choice-free development here,
and focused on the single-sorted case.

\section{The locally ordered category of uniform preorders}\label{se:locat-uord}

Uniform preorders were introduced in~\cite{frey2013fibrational} as
representations of certain {$\Set$-indexed preorders} that generalize
Hofstra's \emph{basic combinatorial objects (BCOs)} \cite{hofstra2006all}.

Contrary to BCOs, for \uords there exists a straightforward characterization of
the induced class of indexed preorders, which makes the notion both conceptually
very clear and somewhat tautological. In this section we reconstruct the
definition of uniform preorders from this characterization, after fixing
terminology and notation on {locally ordered categories} and {indexed
preorders}, which constitute the central formalisms in this article.

A \emph{$\Set$-indexed preorder} is a pseudofunctor $\setord$ where $\Ord$ is
the locally ordered category of preorders and monotone maps. We view locally
ordered categories as degenerate $2$-categories, and use $2$-categorical
concepts and terminology. As we only consider indexed preorders on $\Set$ in
this paper, we omit the prefix. Given an indexed preorder $\fifp$ and a set $A$, we call
$\fifp(A)$ the \emph{fiber} of $\fifp$ over $A$, and refer to its elements as
\emph{predicates} on $A$. Given a function $f:A\to B$, the monotone map
$\fifp(f)$ is called \emph{reindexing along $f$} and abbreviated $f^*$. We
write $\IOrd$ for the locally ordered category of indexed preorders and
pseudo-natural transformations.

\emph{Strict} indexed preorders and 
transformations form a non-full locally ordered subcategory $[\Set\op,\Ord]$ of
$\IOrd$, which by a well known argument about models of geometric theories in
presheaf categories\footnote{\cite[Corollary~D1.2.14(i)]{elephant2} gives a
statement for small index categories, but smallness is not essential.} is
isomorphic to the locally ordered category $\Ord([\Set\op,\Set])$ of internal
preorders in $[\Set\op,\Set]$.

The locally ordered category $\UOrd$ of uniform preorders is now characterized
as fitting into the following strict pullback of locally ordered categories,
where $U$ sends internal preorders to underlying presheaves, the categories in
the lower line are viewed as having codiscretely ordered hom-sets (to make $U$
well-defined), $\yo$ is the Yoneda embedding, and $\famf$ is the indicated
composition.
\begin{equation}\label{eq:uord-pb}
\begin{tikzcd}[sep = small]
|[label={[overlay, label distance=-2.5mm]-45:\lrcorner}]| \UOrd
        \ar[d]
        \ar[r,"J"']
        \ar[rr, "\famf", rounded corners, 
            to path={ -- ([yshift=7pt]\tikztostart.north) 
                      -| (\tikztotarget) [near start]\tikztonodes }]
& \Ord([\Set\op,\Set])
        \ar[d,"U"]
        \ar[r,"\cong"']
&  {[\Set\op,\Ord]}
        \ar[r,hook]
&  \IOrd
\\ \Set
        \ar[r,"\yo"]
& \left[\Set\op,\Set\right]
\end{tikzcd}\end{equation} 
The 2-functor $J$ is 2-fully faithful since $\yo$ is, which means that $\UOrd$
can be identified with the 2-full subcategory of $\Ord([\Set\op,\Set])$ on
internal preorders whose underlying presheaves are representable. In other
words, a uniform preorder is a set $A$ together with an internal preorder
structure on $\yo(A)$. Such a preorder structure is given by a subfunctor of
$\yo(A)\x\yo(A)\cong\yo(A\x A)$, \ie a sieve on $A\times A$, subject to
reflexivity and transitivity conditions.

Since surjections split in $\Set$, sieves are completely determined by their
monomorphisms, or equivalently subset-inclusions, which means that a sieve on
$A\times A$ is equivalently represented as a down-closed subset of the powerset
$P(A\times A)$. We leave it to the reader to verify that unwinding the meaning
of reflexivity, transitivity, monotonicity, and the hom-set ordering in terms of
this representation of sieves yields the following concrete descriptions of the
locally ordered category $\UOrd$ and the 2-functor $\famf$.
\begin{definition}\label{def:uord}
The locally ordered category $\UOrd$ of uniform preorders and monotone maps
is defined as follows.
\begin{enumerate}
\item\label{def:uord-uord} 
A \emph{uniform preorder} is a pair $\brar$ of a set $A$ and a set 
$R \subs P(A\times A)$ of binary relations on $A$, such that
\begin{itemize}
\item $\id_A\in R$,
\item $s\circ r\in R$ whenever $r\in R$ and $s\in R$, and 
\item $s\in R$ whenever $r\in R$ and $s\subs r$.
\end{itemize}
\item\label{def:uord-mmap}
 A \emph{monotone map} between uniform preorders $\brar$ and $\brbs$ is a 
function $f:A\to B$ such that for all $r\in R$, the set 
\begin{equation}
(f\times f)[r]
\;=\; f \circ r \circ f^\circ
\;=\; \setof{(fa,fa')}{(a,a')\in r}
\end{equation}
is in $S$.
\item\label{def:uord-hom_ord}
The ordering relation $\leq$ on monotone maps  $f,g:\brar\to\brbs$ is
defined by $f\leq g$ iff the set
\begin{equation}
\image\angs{f,g}\;=\;\setof{(fa,ga)}{a\in A}
\end{equation}
is in $S$.
\end{enumerate}
\end{definition}
\begin{definition}\label{def:fam}
The 2-functor $\famf:\UOrd\to\IOrd$ 
is defined as follows.
\begin{enumerate}
\item\label{def:fam-obj}
For every \uord $\brar$, the \iord $\famf\brar$ maps
\begin{itemize}
\item\label{def:fam-mor}
sets $I$ to preorders $(A^I,\leq)$, where
$\varphi\leq\psi:I\to A$ iff
\begin{equation}\label{eq:unif-order}
\image\angs{\varphi,\psi}\;=\;\setof{(\varphi i,\psi i)}{i\in I}
\end{equation}
is in $R$, and 
\item functions $f:J\to I$ to monotone maps
$f^*:(A^J,\leq)\to(A^I,\leq)$ given by precomposition.
\end{itemize}
\item\label{def:uord-fam-mor} For every monotone map $f:\brar\to\brbs$ between
indexed preorders, the components of the indexed monotone map
$\famf(f):\famf\brar\to\famf\brbs$ are given by postcomposition.
\end{enumerate}
\end{definition}
\begin{remarks}
\begin{itemize}
\item Given a uniform preorder $\brar$ and predicates, $\vf,\psi:I\to A$, we say that
a relation $r\in R$ \emph{realizes} an inequality $\vf\leq\psi$ if
$\image\langle\vf,\psi\rangle\subs r$ (and thus $\image \langle \vf,\psi
\rangle\in R$). This is stable under reindexing: if $r$ realizes $\vf\leq\psi$
and $u:J\to I$ then $r$ realizes $u^*\vf\leq u^*\psi$.
\item The ordering on monotone maps $f,g:\brar\to\brbs$ defined
in~\ref{def:uord}\ref{def:uord-hom_ord} is the restriction of the ordering on
$\famf\brbs(A)$ as defined in~\ref{def:fam}\ref{def:fam-obj}.
\end{itemize}
\end{remarks}
\begin{definition}\label{def:basis}
A \emph{basis} for a uniform preorder $\brar$ is a subset $R_0\subs R$ of binary
relations whose \emph{down-closure} $\downcl R_0$ in $P(A\times A)$ is $R$, \ie
$R$ and $R_0$ generate the same sieve on $A\times A$. In other words, $R_0\subs
R$ is a basis of $R$ if for every $r\in R$ there is an $r_0\in R_0$ with $r\subs
r_0$. 
\end{definition}
\begin{remark}\label{rem:basis}
Given a set $A$ and a set $R_0\subs P(A\times A)$ of binary relations, 
its down-closure $R=\downcl R_0$ is a \uord structure on $A$ iff 
\begin{enumerate}
\item there exists an $r\in R_0$ with $\id_A\subs r$, and 
\item for all $r,s\in R_0$ there exists a $t\in R_0$ with $s\circ r\in t$.
\end{enumerate}
Just like continuity of functions between topological spaces, monotonicity of
functions between uniform preorders can be expressed in terms of bases.
Specifically, given uniform preorders $\brar$ and $\brbs$ with bases $R_0$ and
$S_0$, a function $f:A\to B$ is monotone iff for all $r\in R_0$ there exists an
$s\in S_0$ with $(f\times f)[r]\subs s$, and given $\varphi,\psi:I \to A$ we
have $\varphi\leq\psi$ in $\famf\brar(I)$ iff there exists an $r\in R_0$ with
$\image\angs{\varphi,\psi}\subs r$.
\end{remark}
The following lemma gives a better understanding of the combined embedding from
$\UOrd$ to $\IOrd$. Recall that a \emph{generic predicate} in an indexed
preorder $\fifa$ is a predicate $\iota\in\fifa(A)$ for some $A$, such that for
every other set $B$ and predicate $\varphi\in\fifa(B)$ there exists a function
$f:B\to A$ with $f^*\iota\cong\varphi$.
\begin{lemma}\label{lem:genpred-image}
The $2$-functor $\famf:\UOrd\to\IOrd$ is a local equivalence, and its
bi-essential image consists of the indexed preorders which admit a generic
predicate. 

Concretely, if $\fifh$ is an indexed preorder with generic predicate
$\iota\in\fifh(A)$, then the corresponding uniform preorder is given by $(A,R)$
with 
\begin{equation}\label{eq:pq}
R=\setof{r\subs A\x A}{p^*\iota\leq q^*\iota}\qquad\qquad
\begin{tikzcd}
&   r       \ar[d,hook]
            \ar[dl,"p"']
            \ar[dr,"q"]
\\  A
&   A\x A
    \ar[r,"\pi_2" near start]
    \ar[l,"\pi_1"' near start]
&   A
\end{tikzcd}
\end{equation}
where $p,q:r\to A$ are the first and second projections as in the diagram.
\end{lemma}
\begin{proof}
For the first claim --- since $\UOrd\to[\Set\op,\Ord]$ is an isomorphism on
$\hom$-preorders, and $[\Set\op,\Ord]\to\IOrd$ is locally order reflecting ---
it is sufficient to show that for every uniform preorder $\brar$, strict indexed
preorder $\fifk$, and pseudonatural $f:\famf\brar\to\fifk$ there exists a strict
transformation $\bar f:\famf\brar\to\fifk$ with $\bar f\cong f$. The
transformation $\bar f$ is given by $\bar f_I(\varphi:I\to A) =
\varphi^*(f_A(\id_A))$\footnote{More generally, this argument works for
pseudonatural transformations $f:\fifh\to\fifk$ between strict indexed preorders
where $\fifh$'s underlying presheaf of sets is \emph{projective}, i.e.\ a
coproduct of representables. Such indexed preorders $\fifh$ correspond to the
`many-sorted uniform preorders' studied in~\cite{frey2013fibrational}.}.

For the second claim it is clear that indexed preorders $\famf\brar$ have
generic predicates (the identity), and that this property is stable under
equivalence. Conversely, it was stated earlier that uniform preorders can be
identified with strict indexed preorders whose underlying presheaf of sets is
representable, and every indexed preorder $\fifh$ with generic predicate
$\iota\in\fifh(A)$ is equivalent to the strict indexed preorder with underlying
presheaf $\Set(-,A)$, and ordering on $\Set(I,A)$ given by $f\leq g$ iff
$f^*\iota\leq g^*\iota$. 
\end{proof}
\begin{examples}\label{ex:uords}
\begin{enumerate}
\item\label{ex:uords:canonical-indexing}The \emph{canonical indexing} of a preorder $(A,\leq)$ is the strict indexed
preorder whose underlying presheaf is the representable presheaf $\Set(-,A)$,
and whose fibers are ordered pointwise, i.e.\ $(\vf:I\to A)\leq(\psi:I\to A)$
iff $\fa i\in I\;.\;\vf(i)\leq\psi(i)$.

The corresponding uniform preorder is $(A,R_\leq)$ where $R_\leq = \down
\{\leq\}\subs P(A\x A)$.

\item\label{ex:uords:bco}Hofstra's \emph{basic combinatory objects} (BCOs)~\cite[pg.~241]{hofstra2006all}
can be embedded into uniform preorders: recall that a BCO is a triple
$(A,\leq,\mcf)$ where $(A,\leq)$ is a partial order and $\mcf$ is a set of
monotone partial endofunction with down-closed domains, which is weakly closed
under composition in the sense that
\begin{inumerate}
\item there exists an $i\in\mcf_A$ such that $i(a)\leq a$ for all $a\in A$, and
\item for all $f,g\in\mcf$ there exists $h\in\mcf$ such that $h(a)\leq g(f(a))$
whenever the right side is defined.
\end{inumerate}
Given a BCO $(A,\leq,\mcf)$, we get an indexed preorder structure on
$\Set(-,A)$ by setting
\begin{equation}
(\vf:I\to A)\leq(\psi:I\to A)\qqtext{iff}\ex f\in \mcf\sfa i\in I\qdot f(\vf(i))\leq\psi(i).
\end{equation}
Just as for the indexed preorders associated to ordinary preorders and uniform
preorders, we write $\famf(A,\leq,\mcf)$ for this indexed preorder.

The corresponding uniform preorder structure $R_\mcf$ on $A$ is generated by the
relations $\setof{r_f\subs A\x A}{f\in\mcf}$, where
$r_f=\setof{(a,b)}{f(a)\leq b}$ for $f\in\mcf$. The axioms (i), (ii) ensure that
the relations $r_f$ form a basis in the sense of Definition~\ref{def:basis}.
\end{enumerate}
\end{examples}
Hofstra defined a locally ordered category $\BCO$ of BCOs whose notion of
morphism is a bit subtle, but is justified and fully explained by the fact that
it extends the mapping $(A,\leq,\mcf)\mapsto\famf(A,\leq,\mcf)$ to a $2$-functor
$\famf:\BCO\to[\Set\op,\Ord]$ into strict indexed preorders which is
\emph{2-fully faithful}, i.e.\ a local isomorphism. Since the embeddings of
$\Ord$ and $\UOrd$ into $[\Set\op,\Ord]$ are also local isomorphisms, we obtain
a sequence
\begin{equation}\label{eq:ord-bco-uord-iord}
\Ord\to\BCO\to\UOrd\to[\Set\op,\Ord]
\end{equation}
of $2$-full embeddings of locally ordered categories.

\section{Adjunctions of uniform preorders}\label{se:uord-adj}

An {adjunction} in a \loc $\mfa$ is a pair of arrows $f:A\to B$, $g:B\to A$,
such that $\id_A\leq g\circ f$ and $f\circ g\leq \id_B$. 
Since $\UOrd\to\IOrd$ is a local equivalence, a monotone map
$f:\brar\to\brbs$ has a right adjoint in $\UOrd$ precisely if $\famf(f)$ has
a right adjoint in $\IOrd$. The following lemma gives a criterion for the
existence of right adjoints in which monotonicity does not have to be checked
explicitly.

\begin{lemma}\label{lem:adj-cond} \tfae for uniform preorders $\brar$, $\brbs$,
a monotone map $f:\brar\to\brbs$, and a function $g:B\to A$.
\begin{enumerate}
\item\label{lem:adj-cond-uord} The function $g$ is a monotone map from $\brbs$ to $\brar$, and right adjoint to
$f$.
\item\label{lem:adj-cond-cond} \begin{enumerate}[label=\normalfont(\arabic*)]
\item\label{lem:adj-cond-cond-counit}The relation $\image\angs{f\circ g,\id_B}=\setof{(f(g(b)),b)}{b\in B}$ is in
$S$, and 
\item\label{lem:adj-cond-cond-trans}for all $s\in S$, the relation $s^*
= \setof{(a,gb)}{(fa,b)\in s}$ is in $R$.
\end{enumerate}
\end{enumerate}
If $\brbs$ is given by a basis, then it is sufficient to verify
\ref{lem:adj-cond-cond-trans} on the elements of the basis.
\end{lemma}
\begin{proof}
First assume \ref{lem:adj-cond-uord}. Condition~\ref{lem:adj-cond-cond-counit}
is equivalent to $f\circ g\leq \id_B$ by~\eqref{eq:unif-order}. For
condition~\ref{lem:adj-cond-cond-trans}, let $I=\setof{(a,b)\in A\times
B}{(fa,b)\in s}$, and let $p:I\to A$ and $q:I \to B$ be the projections. Then we
have $f\circ p\leq q$ in $\famf\brbs(I)$ by direct verification, and therefore
$p\leq g\circ q$ in $\famf\brar(I)$ by exponential transpose. the latter 
is equivalent to the claim.

Conversely, assume \ref{lem:adj-cond-cond}. To see that postcomposition with $g$
induces a left adjoint to $\famf(f):\famf\brar\to\famf\brbs$, it is enough to
check that for all sets $I$ 
and $h:I\to B$, the function $g\circ h$ is a
greatest element of 
\begin{equation}
\Phi=\setof{k:I\to A}{f\circ k \leq h}\subs \famf\brar(I).
\end{equation}
We have $g\circ h\in \Phi$ by \ref{lem:adj-cond-cond-counit}. 
To show that it is a greatest element we have to show that 
$f\circ k \leq h$
implies 
$k\leq g\circ h$, 
which follows from 
\ref{lem:adj-cond-cond-trans} since 
\begin{equation}
\image\angs{k,g\circ h}\subs\image\angs{f\circ k,h}^*
\end{equation}
and $R$ is down-closed.
\end{proof}

\section{Cartesian uniform preorders}\label{se:cart-uord}

The full subcategory of $\IOrd$ on indexed preorders admitting a generic
predicate is closed under small $2$-products: if $\fami{\fifh_k}{k\in K}$ is
a family of indexed preorders with generic predicates
$\fami{\iota_k\in\fifh_k(A_k)}{k\in K}$, then a generic predicate of the
(pointwise) product $\prod_{k\in K}\fifh_k$ is given by the family
\begin{equation}
\textstyle\fami{\pi_k^*\iota_k}{k\in K}\in\prod_{k\in K}\fifh_k(\prod_{k\in K}A_k).
\end{equation}
Thus, $\UOrd$ has products which are preserved by $\famf:\UOrd\to\IOrd$.
Concretely, the terminal uniform preorder is the singleton set with the unique
uniform preorder structure, and a product of $\brar$ and $\brbs$ is given by
$(A\x B,R\ox S)$, where $R\ox S$ is the uniform preorder structure generated by
the basis $\setof{r\x s}{r\in R,s\in S}$. 
\begin{definition}
An object $A$ of a locally ordered category $\mfa$ with finite $2$-products is
called \emph{cartesian} if the terminal projection $A\to 1$ and the diagonal
$A\to A\x A$ have right adjoints $\top:1\to A$ and $\wedge:A\to A\x A$.

Given cartesian objects $A,B$, a \emph{morphism} $f:A\to B$ is called cartesian
if the diagrams 
\begin{equation}
\begin{tikzcd}
    A\x A
        \ar[r,"f\x f"']
        \ar[d,"\wedge"']
&   B\x B
        \ar[d,"\wedge"]
\\  A
        \ar[r,"f"]
&   B
\end{tikzcd}
\qquad\qquad
\begin{tikzcd}
    1       \ar[dr,"\top"]
            \ar[d,"\top"']
&   {}
\\  A       \ar[r,"f"]
&   B
\end{tikzcd}
\end{equation}
commute up to isomorphism.
\end{definition}
Since $\UOrd\to\IOrd$ is a local equivalence and preserves (finite)
$2$-products, a uniform preorder $\brar$ is cartesian if and only if
$\famf\brar$ is cartesian, and the latter is easily seen to be equivalent to
$\famf\brar$ being an \emph{indexed meet-semilattice}, i.e.\ an indexed preorder
whose fibers have finite meets, which are preserved by reindexing. Instantiating
Lemma~\ref{lem:adj-cond} we get the following characterization.
\begin{lemma}\label{lem:cart-cond}A uniform preorder $\brar$ is cartesian if and only if there exists a function
$\wedge:A\times A\to A$ and an element $\top\in A$ such that the relations
\begin{equation}
       \tau     = \setof{(a,\top)}{a\in A}
\qquad \lambda  = \setof{(a\wedge b,a)}{a,b\in A}
\qquad \rho     = \setof{(a\wedge b,b)}{a,b\in A}
\end{equation}
are in $R$, and for all $r,s\in R$ the relation
\begin{equation}
\llangle r,s\rrangle
\;:=\;\wedge\circ(r\x s)\circ\delta_A
\;=\;\setof{(a,b\wedge c)}{(a,b)\in r,(a,c)\in s}
\end{equation}
is in $R$.
\qed
\end{lemma}
\begin{examples}\label{ex:cuords}
\begin{enumerate}
\item\label{ex:cuords:ord} The canonical indexing of a preorder $(A,\leq)$ is an
indexed meet-semilattice if and only if $(A,\leq)$ is an meet-semilattice if and
only if the uniform preorder $(A,\down\{\leq\})$ is cartesian. This follows
since $\Ord\to\UOrd$ is $2$-fully faithful and preserves finite $2$-products.
\item\label{ex:cuord:prim} The primitive recursive functions $f:\N\to\N$ form a
basis (Definition~\ref{def:basis}) for a uniform preorder structure on $\N$
which is cartesian: $\top$ is given by $0$ (or any other number), and a meet
operation $\wedge:\N\x\N\to\N$ is given by any primitive recursive pairing
function.
\item\label{ex:cuords:rec} Instead of primitive recursive function, we can use
total recursive, or even partial recursive functions in the previous example.
The last option gives an instance of the concept of \emph{partial combinatory
algebra}, to which we will come back later.
\end{enumerate}
\end{examples}
\begin{remark}
The forgetful functor from cartesian uniform preorders to uniform preorders does
not have a left biadjoint. This is because the meet-completion of an indexed
preorder with generic predicate does generally not have a generic predicate. The
situation is different for existential quantification, which we treat next.
\end{remark}

\section{Existential quantification}\label{se:exl-quant}

\begin{definition}\label{def:ex-iord}
\begin{enumerate}
\item\label{def:ex-iord:has-ex}We say that an indexed preorder $\fifh$ \emph{has existential quantification},
if for every function $u:J\to I$, the monotone map $u^*:\fifh(I)\to\fifh(J)$ has
a left adjoint $\ex_u:\fifh(J)\to\fifh(I)$, and the \emph{Beck--Chevalley
condition} holds: for every pullback \begin{equation}
\begin{tikzcd}[sep = small]
|[label={[overlay, label distance=-2.5mm]-45:\lrcorner}]| L
        \ar[r,"\bar u"]
        \ar[d, "\bar v"']
&   K
        \ar[d, "v"]
\\  J
        \ar[r, "u"]
&   I
\end{tikzcd}
\end{equation}
in $\Set$ we have $u^*\circ\ex_v\cong \ex_{\bar v}\circ \bar u^*$. 
\item\label{def:ex-iord:commutes-ex}We say
that an indexed monotone map $f:\fifh\to\fifk$ \emph{commutes with existential
quantification}, if $f_I\circ \ex_u\cong\ex_u\circ f_J$ for all $u:J\to I$.

We write $\ex\h\IOrd$ for the sub-2-category of $\IOrd$ on indexed preorders
with existential quantification, and indexed monotone maps commuting with
existential quantification, and we write $\ex\h\UOrd$ for the corresponding
sub-2-category of $\UOrd$, given by the following pullback.
\begin{equation}
\begin{tikzcd}[sep = small]
|[label={[overlay, label distance=-2.5mm]-45:\lrcorner}]|
    \ex\h\UOrd
    \ar[r,""']
    \ar[d,hook]
&\ex\h\IOrd
    \ar[d,hook]
\\ \UOrd
    \ar[r,"\famf"]
& \IOrd
\end{tikzcd}
\end{equation}

\item\label{def:ex-iord:ex-compl} An indexed monotone map $f:\fifa\to\fifh$ from
an indexed preorder $\fifa$ to an indexed preorder $\fifh$ with existential
quantification is called an \emph{\ecompletion}, if for all indexed preorders
$\fifk$ with existential quantification, the precomposition map
\begin{equation}
(-\circ f)\;:\;\hex\IOrd(\fifh,\fifk)\sto\IOrd(\fifa,\fifk)
\end{equation}
is an equivalence of preorders.
\item\label{def:ex-iord:ex-prime} Given a uniform preorder $\fifh$ with
existential quantification, a predicate $\pi\in\fifh(I)$ is called
\emph{\eprime} if for all functions $I\xleftarrow{u}J\xleftarrow{v}K$ and
predicates $\varphi\in\fifh(K)$ with $u^*\pi\leq\ex_v\varphi$, there exists a
function $s: J\to K$ such that $v\circ s=\id_J$ and $u^*\pi\leq s^*\varphi$.

We write $\prim(\fifh)$ for the indexed sub-preorder of $\fifh$ on
$\exists$-prime predicates. 

We say that $\fifh$ has \emph{enough} $\exists$-prime predicates if for every
set $I$ and $\vf\in\fifh(I)$ there exists a $u:J\to I$ and a
$\pi\in\prim(\fifh)(J)$ such that $\ex_u\pi\cong\vf$.
\end{enumerate}
\end{definition}
\begin{remark}\label{rem:prime-fib}
 Using the fibrational---rather than the
indexed---point of view, we can give the following characterization of
$\exists$-prime predicates: $\pi\in\fifh(I)$ is $\exists$-prime iff for all
$f:J\to I$, the object $(J,f^*\pi)$ of the total category $\gcons\fifh$ has the
left lifting property w.r.t.\ cocartesian arrows.
\end{remark}
The notion of $\exists$-prime predicate gives rise to a sufficient criterion for
an indexed preorder with existential quantification to be a $\ex$-completion.
\begin{proposition}\label{prop:ecomp-if-enough-primes}Let $\fifh$ be an indexed preorder with existential quantification, and assume
that $\fifa\subs\fifh$ is an indexed sub-preorder such that
\begin{enumerate}
\item all predicates in $\fifa$ are $\exists$-prime in $\fifh$, and 
\item for every set $I$ and predicate $\vf\in\fifh(I)$ there exists a function
$u:J\to I$ and a predicate $\pi\in\fifa(J)$ such that $\vf\cong\ex_u\pi$.
\end{enumerate}
Then the inclusion $\fifa\incl\fifh$ is an $\ex$-completion, and moreover
$\fifa\incl\prim(\fifh)$ is an equivalence, i.e.\ every $\exists$-prime
predicate in $\fifh$ is isomorphic to one in $\fifa$. In particular, if $\fifh$
has enough $\exists$-prime predicates, then $\prim(\fifh)\incl\fifh$ is an
$\ex$-completion.
\end{proposition}
\begin{proof}
Given an indexed preorder $\hypk$ with existential quantification and an indexed
monotone map $f:\fifa\to\hypk$, define $\tilde f:\fifh\to\fifk$ by $\tilde
f_I(\vf)=\ex_uf(\pi)$ for a choice of function $u:J\to I$ and predicate
$\pi\in\fifa(J)$ with $\ex_u\pi\cong\vf$. It is straightforward to verify that
$\tilde f$ gives a well defined indexed monotone map commuting with existential
quantification, and the assignment $f\mapsto\tilde f$ gives a pseudoinverse to
the restriction map $\hex\IOrd(\fifh,\fifk)\sto\IOrd(\fifa,\fifk)$.

Now assume that $\pi\in\prim(\fifh)(I)$, and choose $u:J\to I$ and
$\sigma\in\fifa(J)$ with $\ex_u\sigma\cong\pi$. Then from $\pi\leq\ex_u\sigma$
it follows that there exists a section $s$ of $u$ with $\pi\leq s^*\sigma$. On
the other hand, the inequality $\ex_u\sigma\leq\pi$ is equivalent to $\sigma\leq
u^*\pi$, which implies $s^*\sigma\leq \pi$ by applyinng $s^*$ on both sides, and
we conclude that $s^*\sigma\cong\pi$.
\end{proof}
\begin{definition}
A \emph{primal $\ex$-completion} is an $\ex$-completion $e:\fifa\to\fifh$
fitting the hypotheses of Proposition~\ref{prop:ecomp-if-enough-primes}, i.e.\
$\fifh$ has enough $\exists$-primes and $e$ is equivalent to
$\prim(\fifh)\incl\fifh$.
\end{definition}
It is well known that indexed preorders on \emph{small} index categories $\CC$
always admit primal $\ex$\=/completions\footnote{For accounts of closely related
constructions see e.g.\ \cite[Definition~3.4.5]{frey2013fibrational} for the
$\ex$-completion of fibered preorders satisfying a stack-condition,
\cite[Section~4]{trotta2020existential} for $\ex$-completion of indexed
meet-semilattices, and Hofstra~\cite[Section~3.2]{hofstra2011dialectica} for the
analogous construction for non-posetal fibrations.}: given an indexed preorder
$\fifa:\CC\op\to\Ord$, predicates on $I\in\CC$ in its $\ex$-completion
$D\fifa:\CC\op\to\Ord$ are given by pairs $(J\xto{u}I,\vf\in\fifa(J))$, where
$(J\xto{u}I,\phi)\leq(K\xto{v}I,\psi)$ iff there exists a $w:J\to K$ such that
$v\circ w=u$ and $\vf\leq w^*\psi$. However, for indexed preorders on $\Set$
this construction may not be well-defined, since the resulting indexed preorder
may have large fibers. In the following we show that indexed preorders arising
from uniform preorders \emph{do} always admit primal $\ex$-completions, which
are again representable by uniform preorders (the question if there are
non-primal $\ex$-completions over $\Set$ remains open).
\begin{definition}\label{def:dar}
For $\brar$ a uniform preorder, we define the uniform preorder 
\begin{equation}
D\brar\,=\,(PA,DR)
\end{equation}
where $PA$ is the powerset of $A$, and $DR$ is the uniform preorder structure on
$PA$ generated by the basis of relations 
\begin{equation}
[r]\,=\, \setof{(U,V)\in PA\x PA}{\fa
a\in U\sex b\in V\;.\; (a,b)\in r}
\end{equation}
for $r\in R$.
\end{definition}
\begin{remarks}\label{rem:dar}\begin{enumerate}
\item\label{rem:dar:is-basis}The relations $[r]$ do indeed constitute a basis since
$\id_{PA}\subs[\id_A]$ and $[s]\circ [r]\subs[s\circ r]$ for $r,s\in R$.
\item\label{rem:dar:fiberwise-order}Unwinding the definition of $D\brar$ we see that for $\varphi,\psi:I\to PA$ we
have $\varphi\leq\psi$ in $\famf(D\brar)(I)$ if and only if there exists an
$r\in R$ such that 
\begin{empheq}{equation}
\fa i\in I \sfa a\in \varphi(i)\sex b\in\psi(i)\qdot (a,b)\in r.
\end{empheq}
\end{enumerate}
\end{remarks}
\begin{proposition}\label{prop:dar-primal-ecomp}For every uniform preorder $\brar$, the indexed preorder $\famf(D\brar)$ has
existential quantification and the singleton map $\eta:A\to PA$ is monotone
from $(A,R)$ to $D(A,R)$. The induced indexed monotone map
$\famf(\eta):\famf\brar\to\famf(D\brar)$ is a primal $\ex$-completion.

\end{proposition}
\begin{proof}
Existential quantification in $\famf(D\brar)$ is given by union, i.e.\
\begin{equation}
\textstyle(\ex_u\vf)(i)=\bigcup_{u(j)=i}\vf(j)
\end{equation}
for $u:J\to I$ and $\vf:J\to PA$, and $\eta$ is monotone since for every $r\in
R$ we have 
\begin{equation}
\setof{(\{a\},\{a'\})}{(a,a')\in r}\,\subs\,[r].
\end{equation}
To show that $\famf(\eta)$ is a primal $\ex$-completion it remains to show that
it is fiberwise order reflecting, and its image in $\famf(D\brar)$---the indexed
sub-preorder of \emph{singleton-valued predicates}, i.e.\ predicates factoring
through $\eta:A\to PA$---satisfies the hypotheses of
Proposition~\ref{prop:ecomp-if-enough-primes}.

The fact that $\famf(\eta)$ is order reflecting follows immediately from the
explicit description of the fiberwise ordering in $\famf(D\brar)$ in
Remark~\ref{rem:dar}\ref{rem:dar:fiberwise-order}.

To see that singleton-valued predicates are $\exists$-prime in $\famf(D\brar)$,
assume $\vf:I\to A$, $\psi: J\to PA$, and $u:J\to I$ such that $\eta\circ
\vf\leq\ex_u\psi$. Unwinding definitions this means that there exists an $r\in
R$ such that 
\begin{equation}
\textstyle\fa i\in I\sfa a\in \{\vf(i)\}\sex b\in\bigcup_{u(j)=i}\psi(j)\;.\; 
(a,b)\in r\,,\, 
\end{equation}
i.e.\ 
\begin{equation}
\textstyle\fa i\in I\sex j\in J\;.\; u(j)=i\wedge\ex b\in \psi(j)\;.\;
    (\vf(i),b)\in r\,,\, 
\end{equation}
and the required section of $u$ is given by a Skolem function for the first two
quantifiers. 

Finally, $\famf(D\brar)$ has `enough' singleton-valued predicates, since every
predicate $\vf:I\to PA$ can be written as $\vf=\ex_u\sigma$ for $J=\coprod_{i\in
I}\vf I$, $u$ the first projection, and
$\sigma=(J\stackrel{\pit}{\to}A\stackrel{\eta}{\to}PA)$.
\end{proof}
\begin{remark}\label{rem:d-lax-idempotent} The assignment $\brar\mapsto D\brar$
gives rise to a left $2$-adjoint to the inclusion $\ex\h\UOrd\to\UOrd$, and the
unit $\eta$ and multiplication $\mu$ of the induced $2$-monad $D:\UOrd\to\UOrd$
are componentwise given by singleton map and union. The $2$-monad is \emph{lax
idempotent}\footnote{Lax idempotent monads were introduced
in~\cite{zoeberlein1976doctrines, kock1995monads} and are also known as
\emph{Kock-Zöberlein monads}. (The articles were published 19 years apart, but
Kock's preprint seems to have been contemporaneous with Zöberlein's thesis, on
which his article is based. The name \emph{lax idempotent} is due to Zöberlein
and was later picked up by Kelly and Lack~\cite{kelly1997property}.)} in the
sense that $D\eta_{\brar}\adj\mu_\brar\adj\eta_{D\brar}$ for all uniform
preorders $\brar$. In particular, a uniform preorder $\brar$ is a $D$-algebra
iff $\eta_\brar$ has a left adjoint (the adjunction is then automatically a
reflection, since $\famf(\eta_\brar)$ is fiberwise order-reflecting). Finally,
the adjunction is monadic, since reflective indexed sub-preorders of indexed
preorders with existential quantification have existential quantification. 
\end{remark}

\section{Indexed frames}\label{se:ifrms}

We recall the definition of indexed frames from~\cite{frey2023categories}.\begin{definition}
An \emph{indexed frame} is an indexed meet-semilattice $\fifh$ which has
existential quantification and moreover satisfies the \emph{Frobenius
condition}: for all functions $u:J\to I$, and predicates $\varphi\in\fifh_I$ and
$\psi\in\fifh_J$ we have $\varphi\wedge\ex_u\psi\cong\ex_u(u^*\vf\wedge\psi)$.
\end{definition}
\begin{examples}
\begin{enumerate}\label{ex:iframes}
\item
The canonical indexing of a poset $(A,\leq)$ is an indexed frame if and only if
$A$ is a {frame}~\cite{Picado2012frames}, i.e.\ a complete lattice satisfying
the infinitary distributive law $a\wedge\bigvee_ib_i=\bigvee_ia\wedge b_i$.
\item If $(L,\leq)$ is a frame and $M$ is a monoid of frame-endomorphisms (i.e.\
monotone maps preserving finite meets and arbitrary joins), we obtain an indexed
frame structure on the representable functor $\Set(-,L)$ by setting
\begin{equation}
\vf\leq\psi\qqtext{if and only if}\ex m\in M\sfa i\in I\;.\; 
m(\vf(i))\leq\psi(i)
\end{equation}
for $\vf,\psi:I\to L$. This indexed frame structure is only representable by an
ordinary frame if $M$ has a least element (which is then an `interior operator'
i.e.\ a posetal comonad). A non-trivial example is the \emph{Lipschitz
hyperdoctrine} which has been recently proposed by Reid Barton and Johann
Commelin, and is obtained by taking $L=([0,\infty],\geq)$ and $M=\R_{>0}$ acting
by multiplication. See also~\cite{figueroa2022topos} for similar constructions
of non-$\Set$-based indexed preorders.
\end{enumerate}
\end{examples}
Another way of producing indexed frames is given by the following.
\begin{proposition}\label{prop:dar-ifrm} If $\brar$ is cartesian then so are
$D\brar$ and $\eta:\brar\to D\brar$, and moreover $\famf(D\brar)$ is an indexed
frame.
\end{proposition}
\begin{proof}
To show that $D\brar$ is cartesian we use Lemma~\ref{lem:cart-cond} and define
$\wedge:PA\x PA\to PA$ and $\top\in PA$ by  $U\wedge V =\setof{a\wedge b}{a\in
U,b\in V}$ and $\top = \{\top\}$. Then the verification of the conditions is
straightforward.
\end{proof}

\section{\Ccompness}\label{se:comcomp}

\begin{definition}\label{def:uquant}
\begin{enumerate}
\item 
We say that an indexed preorder \emph{has universal quantification} if it
satisfies the dual condition of
Definition~\ref{def:ex-iord}\ref{def:ex-iord:has-ex}. 
\item A \emph{Heyting preorder}\footnote{Contrary to the better-known
\emph{Heyting algebras}, Heyting preorders need not have finite joins---those
will turn out to exist in the cases we're interested in, but we don't have to
assume them.} is a meet-semilattice $(H,\leq)$ which is cartesian closed, i.e.\
for all $a\in H$ the monotone map $(-\wedge a)$ has a right adjoint $(a\imp -)$
called \emph{Heyting implication}.
\item
An \emph{indexed meet-semilattice} $\fifh$ is said to \emph{have implication}
if its fibers are Heyting preorders, and this structure is preserved up to
isomorphism by reindexing.
\item A \emph{tripos} is an indexed meet-semilattice $\trip$ which has universal
quantification, implication, and a generic predicate.
\end{enumerate}
\begin{remarks}
\begin{enumerate}
\item Since they're assumed to have generic predicates, all triposes are
representable by uniform preorders.
\item As explained in~\cite[Theorem~1.4]{hjp80}, using Prawitz-style second
order encodings~\cite[page~67]{prawitz1965natural} one can show that triposes
have existential quantification and fiberwise finite joins which are stable
under reindexing. In other words, triposes are models of full first order logic.
\end{enumerate}
\end{remarks}
\end{definition}
\begin{definition}\label{def:rcomp} A cartesian uniform preorder $(A,R)$ is
called \emph{\ccomp} if there exists a relation $@\in R$ (called
`universal relation'), such that for every relation $r\in R$ there exists a
\emph{function} (i.e.\ a single-valued and entire relation) $\tilde r\in R$ with
\begin{equation}
r\circ\wedge\;\subs\; @\circ \wedge\circ (\tilde r\x \id_A),
\end{equation}
in other words
\begin{equation}\label{eq:impl-rel-compl}
\fa a\,b\,c\in A\;.\;(a\wedge b,c)\in r\;\imp\; (\tilde r(a)\wedge b,c)\in @.
\end{equation}
\end{definition}
\begin{remarks}\label{rem:comcomp}\begin{enumerate}
\item\label{rem:comcomp:snm}\Ccompness can be viewed as a generalization of the {functional completeness}
property of recursive functions expressed by the \emph{s-m-n theorem}, which in
its most basic form (see e.g.~\cite[Theorem~4.4.1]{cutland1980computability})
says that for every partial recursive function $f(x,y)$ in two arguments there
exists a \emph{total recursive} function $\tilde f(x)$ in one argument such that
the partial functions $f(x,y)$ and $\phi_{\tilde f(x)}(y)$ are equal, where
$(\phi_n)_\nin$ is a effective enumeration of partial recursive functions. 

Note that besides using relations instead of partial functions, the statement
above is somewhat weaker than that of the s-m-n theorem since \emph{equality} of
partial functions is replaced by \emph{inclusion} of relations. See also
Remark~\ref{rem:rpca}\ref{rem:rpca:faber}.
\end{enumerate}
\end{remarks}
\begin{theorem}\label{thm:rel-compl}
The following are equivalent for a cartesian uniform preorder $\brar$.
\begin{enumerate}
\item $\brar$ is {\ccomp}.
\item $\famf(D\brar)$ is a tripos.
\end{enumerate}
\end{theorem}
\begin{proof}
Assume first that $\famf(D\brar)$ is a tripos, and assume w.l.o.g.\ that
conjunction is given `on the nose' by the pointwise construction $U\wedge V =
\setof{u\wedge v}{u\in U, v\in V}$ from the proof of
Proposition~\ref{prop:dar-ifrm}. Let $E\;\incl\;A\x A\x P(A\x A)$ be the
membership relation, and define $u:E\to P(A\x A)$ and $\vf,\psi:E\to PA$ by 
\begin{equation}
u(b,c,s)=s\qquad 
\varphi(b,c,s) = \{b\} \qquad 
\psi(b,c,s)=\{c\}.
\end{equation}
We set $\theta=\fa_u(\varphi\imp\psi) : P(A\x A)\to PA$ and let $@\in R$ such
that $[@]$ is a realizer of $u^*\theta\wedge\varphi\leq\psi$. Now for every
$r\in R$ we construct a pullback
\begin{equation}
\begin{tikzcd}
|[label={[overlay, label distance=-2.5mm]-45:\lrcorner}]| 
M
        \ar[r,"x"]
        \ar[d,"v"']
& E
        \ar[d,"u"]
\\ A
        \ar[r,"w"']
&   P(A\x A)
\end{tikzcd}
\qquad\qquad
\begin{aligned}
M &= \setof{(a,b,c)}{(a\wedge b,c)\in r}
\\ v(a,b,c) &= a
\\ x(a,b,c) &= (r,b,c)
\\ w(a)     &= \setof{(b,c)}{(a\wedge b,c)\in r}
\end{aligned}
\end{equation}
and a simple argument using the Beck--Chevalley condition gives $\eta\leq
w^*\theta$, where $\eta:A\to PA$ is the singleton map. Any $s\in R$ such that
$[s]$ realizes this inequality is total, and using choice we pick $\tilde r$ to
be a subfunction, so that $\fa a\in A\;.\; \tilde r(a)\in\theta(w(a))$.
Implication~\eqref{eq:impl-rel-compl} follows since $[@]$ is a realizer of the
inequality $v^*w^*\theta\wedge x^*\vf\leq x^*\psi$.

\medskip

Conversely, assume that $\brar$ is \ccomp. Instead of
constructing implication and universal quantification separately, we show how to
define the `synthetic' connective $\forall_u(\varphi\imp\psi)$ for $u:J\to I$
and $\varphi,\psi\in \famf(D\brar)(I)$. Implication and universal quantification
can then be recovered by either replacing $u$ by the identity, or $\varphi$ by
the true predicate. For $\varphi,\psi:J\to PA$ define
$\forall_u(\varphi\imp\psi):I\to PA$ by 
\begin{equation*}
\forall_u(\varphi\imp\psi)(i)=\bigcap_{uj=i}\{a\in A\msep\forall
b\in\varphi(j)\,\exists c\in\psi(j)\qdot @(a\wedge b,c)\}.
\end{equation*}
It is then easy to see that the inequality
$u^*\forall_u(\varphi\imp\psi)\wedge\varphi\leq\psi$ is realized by $@$; and if
$\zeta:I\to PA$ such that the inequality $u^*\xi\wedge\varphi\leq \psi$ is
realized by $r\in R$, then $\tilde{r}$ realizes
$\xi\leq\forall_u(\varphi\imp\psi)$.
\end{proof}

\begin{remarks}\label{rem:relcompl}
\begin{enumerate}
\item\label{rem:relcompl:equiv-list}The list of equivalent statements in Theorem~\ref{thm:rel-compl} can be extended
by the following, where $\Set[\famf(D\brar)]$ is the \emph{category of partial
equivalence relations and compatible functional relations}\footnote{The
construction of $\Set[\famf(D\brar)]$ from $\famf(D\brar)$ is called \emph{exact
completion of the `existential elementary doctrine'} $\famf(D\brar)$ e.g.\
in~\cite{maietti2012unifying}. If $\famf(D\brar)$ is a tripos, the construction
is the well known \emph{\tttc}~\cite{hjp80}.} in the fibered frame
$\famf(D\brar)$, and $\PAsm\brar\,=\,\gcons(\famf\brar)$ is the total category
of the indexed preorder $\famf\brar$ (which is the classical category of
\emph{partitioned assemblies} if $\brar$ comes from a PCA):
\begin{enumerate}
\item[(iii)] $\Set[\famf(D\brar)]$ is a topos.
\item[(iv)] $\Set[\famf(D\brar)]$ is locally cartesian closed.
\item[(v)] $\PAsm\brar$ is weakly locally cartesian closed.
\end{enumerate}
It is well known that (iii) follows from (ii): this is the reason for the term
`\tttc'. Clearly (iii) implies (iv). Next, (iv) implies (ii) since since every
fibered frame $\fifh$ can be presented as 
\begin{equation}
\fifh\,\simeq\,\bigl(\Set\op\xrightarrow{\Delta\op}
\Set[\fifh]\op\xrightarrow{\mathsf{sub}}\Ord\bigr)
\end{equation}
where $\Delta$ is the constant-objects-functor and $\mathsf{sub}$ is the indexed
preorder of subobjects. If $\Set[\fifh]$ is locally cartesian closed then
$\fifh$ is an indexed Heyting algebra with $\fa$ and $\ex$ since $\mathsf{sub}$
is and this property is stable under precomposition with the finite-limit
preserving functor $\Delta$. In the case $\fifh = \famf(D\brar)$ we furthermore
have a generic predicate, so that $\famf(D\brar)$ is a tripos. 

Finally, the equivalence between (iv) and (v) follows from Carboni--Rosolini's
characterization of locally cartesian closed exact completions (`the exact
completion of a finite-limit category $\mcc$ is locally cartesian closed iff
$\mcc$ is weakly locally cartesian closed',~\cite{carboni2000locally}), since
$\Set[\famf\brar]$ is an ex/lex completion of $\PAsm\brar$ by means of the
functor
\begin{equation}\label{eq:ttt-exlex}
\PAsm\brar\,\to\,\Set[\famf(D\brar)]
\end{equation}
which sends $\vf:I\to A$ to the sub-diagonal p.e.r.\ on $I$ with support
$I\xrightarrow{\vf}A\incl PA$: to verify this fact observe that the functor is
fully faithful, and the objects in its image are projective and cover all other
objects.

Analogous reformulations of \ccompness for \emph{many-sorted}
uniform preorders are given in~\cite[Theorem~4.10.3]{frey2013fibrational}.
\item\label{rem:relcompl:tripos-char}
Theorem~\ref{thm:rel-compl} gives rise to an correspondence between 
\begin{itemize}
\item relational complete uniform preorders $\brar$, and 
\item triposes $\trip$ with enough $\exists$-primes, such that $\prim(\trip)$
has finite meets.
\end{itemize}
If $\brar$ is \ccomp then $\famf(D\brar)$ is such a tripos, and
conversely if $\trip$ is such a tripos, then any $\exists$-prime predicate which
covers the generic predicate of $\trip$ is generic in $\prim(\trip)$, whence the
latter is representable by an uniform preorder $\brar$, which is cartesian by
assumption, and \ccomp by the theorem.

In type theoretic, `univalent' language~\cite{hottbook} one would state this
correspondence as an \emph{equivalence} the \emph{type} of \ccomp
preorders and the \emph{type} of the specified triposes. In classical
foundations this translates into an equivalence of two $1$-groupoids, which can
both be realized as sub-groupoids of the \emph{core}\footnote{i.e.\ the
subgroupoid of all isos} of the hom-wise poset reflection $\IOrd$.
\end{enumerate}
\end{remarks}

\begin{examples}\label{ex:comcomp}
\begin{enumerate}
\item\label{ex:comcomp-slat}
For every meet-semilattice $(A,\leq)$, the uniform preorder $(A,R_\leq)$
corresponding to its canonical indexing $\famf(A,\leq)$  is 
\ccomp. This is because $\famf(D(A,R_\leq))$ is equivalent to the canonical
indexing of the frame of down-sets in $(A,\leq)$, and the latter is known to be
a tripos.
\item\label{ex:comcomp:tripos}
If $\brar$ is an uniform preorder such that $\famf\brar$ is a tripos, then
$\brar$ is \ccomp, and thus $\famf(D\brar)$ is a tripos as well.
This is shown using variant of the construction in the proof of
Theorem~\ref{thm:rel-compl}: we take $E\incl A\x A\x P(A\x A)$ to be the
membership relation as above, let $\vf,\psi:E\to A$ and $u:E\to P(A\x A)$ be
three projections, set $\theta = \fa_u(\vf\imp\psi)$, and take $@$ to be a
realizer of $u^*\theta\wedge\vf\leq\psi$. Given $r\in R$ we again we construct
the pullback
\begin{equation}
\begin{tikzcd}
|[label={[overlay, label distance=-2.5mm]-45:\lrcorner}]| 
M
        \ar[r,"x"']
        \ar[d,"v"']
& E
        \ar[d,"u"]
\\ A
        \ar[r,"w"']
&   P(A\x A)
\end{tikzcd}
\qquad
\begin{aligned}
M &= \setof{(a,b,c)}{(a\wedge b,c)\in r}
\\ v(a,b,c) &= a
\\ x(a,b,c) &= (r,b,c)
\\ w(a)     &= \setof{(b,c)}{(a\wedge b,c)\in r}\,,\,
\end{aligned}
\end{equation}
and chasing around it we get $\id_A\leq w^*\theta$, i.e.\ $\theta\circ w \in R$,
and we take this function to be $\tilde r$. The
implication~\eqref{eq:impl-rel-compl} follows since $@$ realizes
$v^*w^*\theta\wedge x^*\vf\leq x^*\psi$.

Since $\famf\brar$ is a tripos by assumption, $\brar$ is a $D$-algebra, i.e.\
$\eta:\brar\to D\brar$ has a left adjoint $\alpha: D\brar\to\brar$ (see
Remark~\ref{rem:d-lax-idempotent}), and it is easy to verify by hand that this
left adjoint is cartesian. In other words, $\famf\brar$ is a \emph{geometric
subtripos} of $\famf(D\brar)$, and this subtripos inclusion gives rise to a
geometric \emph{subtopos} inclusion $\Set[\famf\brar]\incl\Set[\famf(D\brar)]$
via the \tttc. The intermediate quasitopos of separated objects is the
\emph{q-topos} $\boldsymbol{Q}(\famf\brar)$ associated to the tripos via the
construction described in~\cite[Definition~5.1]{frey2015triposes}. 
We recall that the notion of q-topos is slightly weaker than that of quasitopos
(not requiring coproducts or local cartesian closure), and was introduced
in~\cite{frey2015triposes} since the construction of $\boldsymbol{Q}(\trip)$
does {not} seem to produce a quasitopos over arbitrary base categories. However,
the argument above shows that the construction \emph{does} produce quasitoposes
for $\Set$-based triposes.
\end{enumerate}
\end{examples}
Another large class of examples of \ccomp uniform preorders is
given in the next section.

\section{Ordered partially combinatory algebras}\label{se:opca}

We recall the relevant definitions
from~\cite[Section~2.6.5]{vanoosten2008realizability}.
\begin{definition}
An \emph{ordered applicative structrure} (OPAS) is a triple $(A,\leq,\cdot)$
where $(A,\leq)$ is a poset and $(-\cdot-):A\x A\pto A$ is a partial binary
operation.
\end{definition}
\begin{remarks}
\begin{enumerate}
\item Application associates to the left, i.e.\ $a\ap b\ap c$ is a shorthand for
$(a\ap b)\ap c$.
\item A \emph{polynomial} over an OPAS $(A,\leq,\cdot)$ is a term built up from
variables, constants from $A$, and application $(-\cdot-)$. We write
$p[x_1,\dots,x_n]$ for a polynomial which may (but is not required to) contain
the variables $x_1,\dots,x_n$, and if $a_1,\dots,a_n\in A$ we write
$p[a_1,\dots,a_n]$ for the possibly undefined result of substituting and
evaluating.
\item When reasoning with partial terms, $t\down$ means that $t$ is defined, and
the statement of an equality $s=t$ or inequality $s\leq t$ contains the implicit
assertion that both sides are defined.
\end{enumerate}
\end{remarks}
\begin{proposition}\label{prop:opas-wcombicomp}
The following are equivalent for an OPAS $(A,\leq,\cdot)$.
\begin{enumerate}
\item For all polynomials $p[x_1,\dots,x_n,y]$ over $A$ there exists an element 
$e\in A$ such that for all $a_1,\dots,a_n,b \in A$:
\begin{itemize}
\item $e\ap a_1\ap\dots\ap a_n\down$
\item $p[a_1,\dots,a_n,b]\down$ implies $e\ap a_1\ap\dots\ap a_n\ap b\down$
and $e\ap a_1\ap\dots\ap a_n\ap b\leq p[a_1,\dots,a_n,b]$
\end{itemize}
\item\label{prop:opas-wcombicomp-ks} there exist elements $\comk,\coms\in A$
such that for all $a,b,c\in A$:
\begin{itemize}
\item $\comk\ap a\ap b\leq a$
\item $\coms\ap a\ap b\down$
\item $a\ap c\ap (b\ap c)\down$ implies $\coms\ap a\ap b\ap c\down$ and 
$\coms\ap a\ap b\ap c\leq a\ap c\ap (b\ap c)$ 
\end{itemize}
\end{enumerate}
\end{proposition}
\begin{proof}
\cite[Theorem~1.8.4]{vanoosten2008realizability}
\end{proof}
\begin{definition}\label{def:opca}
\begin{enumerate}
\item\label{def:opca:opca}
An \emph{ordered combinatory algebra} (OPCA) is an OPAS satisfying the
equivalent conditions of Proposition~\ref{prop:opas-wcombicomp}.
\item\label{def:opca:filter} A \emph{filter} on an OPCA is a subset $\Phi\subs
A$ which is upward closed, closed under application, and contains choices of
elements $\comk,\coms$ as in
Proposition~\ref{prop:opas-wcombicomp}\ref{prop:opas-wcombicomp-ks}. A
\emph{filtered OPCA} is a quadruple $(A,\leq,\cdot,\Phi)$ where $(A,\leq,\cdot)$
is an OPCA and $\Phi$ is a filter on $A$.
\end{enumerate}
\end{definition}
Given a filtered OPCA $(A,\leq,\cdot,\Phi)$ we define a strict indexed preorder
structure on the representable presheaf $\Set(-,A)$ by setting 
\begin{equation}
(\vf:I\to A)\leq(\psi:I\to A)\qqdefequi\exists e\in \Phi\sfa i\in 
I\;.\; e\ap\phi(i)
\leq\psi(i)\,,\,
\end{equation}
It follows from standard arguments in combinatory logic that this indexed
preorder is well defined (i.e.\ reflexive and transitive), and actually an
\emph{indexed meet-semilattice}, and as Hofstra explains
in~\cite[p.~252]{hofstra2006all}, its $\ex$-completion is an ordered variant of
a relative realizability construction and in particular a tripos. Thus, the
corresponding uniform preorder $(A,R_\Phi)$
is \ccomp by Theorem~\ref{thm:rel-compl}.

The mapping from filtered OPCAs to uniform preorders factors through BCOs: the
BCO corresponding to $(A,\leq,\cdot,\Phi)$ is given by $(A,\leq,\mcf_\Phi)$,
where 
\begin{equation}
\mcf_\Phi\,=\,\setof{(e\ap-):A\pto A}{e\in\Phi}.
\end{equation}
Thus, a basis for the uniform preorder structure $R_\Phi$ is given by
$\setof{r_e\subs A\x A}{e\in\Phi}$, with $r_e=\setof{(a,b)\in A\x A}{e\ap a\leq
b}$.

\medskip

In the following we describe the \emph{discretely ordered} special case of this
correspondence, which identifies filtered (better known as `relative') PCAs with
\ccomp \emph{discrete combinatory objects}. Since discrete
combinatory objects admit an easy characterization among indexed preorders, this
enables us to give a characterization of (relative) realizability triposes.

\section{Discreteness}\label{se:discrete}

\begin{definition}\label{def:discrete}\begin{enumerate}
\item\label{def:discrete:dco} A \emph{discrete combinatory object (DCO)} is a
uniform preorder where all relations $r\in R$ are \emph{single-valued}, i.e.\
partial functions. We write $\DCO$ for the full locally ordered subcategory of
$\UOrd$ on DCOs.
\item\label{def:discrete:predicate}A predicate ${\delta\in\fifa(I)}$ of an indexed preorder $\fifa$ is called
\emph{discrete} if for every surjection $e:K\epi J$, function $f:K\to I$, and
predicate $\varphi\in\fifa(J)$ such that $e^*\vf\leq f^*\delta$, there exists a
(necessarily unique) $g: J\to I$ with $g\circ e = f$ (and therefore $\vf\leq
g^*\delta$ since reindexing along split epis is order-reflecting).
\end{enumerate}
\end{definition}
\begin{remarks}\label{rem:discrete}
\begin{enumerate}
\item\label{rem:discrete:jpaa-ref}
DCOs were introduced in~\cite[Definition~2.2]{frey2019characterizing} in terms
of bases, i.e.\ as sets $A$ equipped with a set $\mcf$ of partial endofunctions
containing the identify and weakly closed under composition in the sense that
for all $f,g\in\mcf$ there exists an $h\in\mcf$ such that $g\circ f\subs h$.
Down-closure in $P(A\x A)$ of such a structure yields a DCO $(A,\down\mcf)$ in
the above sense inducing the same indexed preorder and the two definitions give
rise to equivalent locally ordered categories, the principal difference being
that for the above, `saturated' definition, the $2$-functor $\DCO\to\IOrd$ is
injective on objects.
\item\label{rem:discrete:total-cat} In fibrational language, discreteness of
$\delta\in\fifa(A)$ says that $(A,\delta)$ has the right lifting property in the
total category $\int\fifa$ w.r.t.\ all  cartesian maps over surjections.
\item\label{rem:discrete:reind-inj-surj} It is easy to see that reindexings of
discrete predicates along injections are discrete again. Reindexings along
surjections, on the other hand, are discrete only in the trivial case that the
surjection is a bijection.
\item\label{rem:discrete:bco}
DCOs embed into BCOs: modulo the issue of bases vs.\ saturated presentations
discussed in~\ref{rem:discrete:jpaa-ref}, they correspond precisely to BCOs
whose order structure is trivial. Thus, we can extend the
sequence~\eqref{eq:ord-bco-uord-iord} of embeddings to the following diagram.
\begin{equation}
\begin{tikzcd}[sep = small]
|[label={[overlay, label distance=-2.5mm]-45:\lrcorner}]| \Set
    \ar[r]
    \ar[d]
& \DCO
    \ar[d]
\\ \Ord
    \ar[r]
& \BCO
    \ar[r]
& \UOrd
    \ar[r]
& {[\Set\op,\Ord]}
    \ar[r]
& \IOrd
\end{tikzcd}
\end{equation}
The intersection of $\Ord$ and $\DCO$ is trivial, in the sense that it only
contains discretely ordered representable presheaves: this is because indexed
preorders representable by ordinary preorders are stacks for the canonical
topology, and if $\famf(A,R)$ is such a stack for a DCO $(A,R)$, then $R$
contains only subfunctions of $\id_A$ (otherwise, the stack condition would give
$(a,a)\leq (a,f(a))$ over $2$).
\end{enumerate}
\end{remarks}
The following clarifies the relationship between the two notions of discreteness
introduced in Definition~\ref{def:discrete}.
\begin{proposition}\label{prop:dco-iff-discrete}
A uniform preorder $\brar$ is a DCO if and only if the generic predicate
$\id_A\in\famf\brar(A)$ is discrete.
\end{proposition}
\begin{proof}
Assume first that $\brar$ is a DCO and consider a span
$J\stackrel{e}{\epil}K\stackrel{f}{\to}A$ with $e$ surjective, and a predicate
$\vf:J\to A$ with $e^*\vf\leq f^*\id_A$. Form the image factorization (1) of
$\langle \vf\circ e,f \rangle$. 
\begin{equation}
\text{(1) }\begin{tikzcd}
    K       \ar[dr,"{\langle \vf\circ e,f \rangle}"']
            \ar[r,"h", two heads]
&   r       \ar[d,"{\langle p,q \rangle}", hook]
\\& A\x A
\end{tikzcd}
\qquad\qquad\text{(2) }\begin{tikzcd}
K
        \ar[r,"h" description, two heads]
        \ar[d,"e"', two heads]
&   r
        \ar[d,"p", tail]
\\  J
        \ar[ru,"k" description, dashed]
        \ar[r,"\vf" description]
&   A
\end{tikzcd}
\end{equation}
Then $r\in R$ and therefore $p$ is injective since $(A,R)$ is a DCO. Since $e$
is surjective we obtain a lifting $k$ in the square (2) and the desired map is
$q\circ k$.

Conversely assume that $\id_A$ is discrete, let $r\in R$, write $\langle p,q
\rangle:r\incl A\x A$ for the inclusion, and let
$r\stackrel{e}{\epi}U\stackrel{m}{\hookrightarrow}A$ be an image factorization
of $p$. We have $p^*(\id_A)=e^*(m^*(\id_A))\leq q^*(\id_A)$, and discreteness of
$\id_A$ implies that there exists $g:U\to A$ with $g\circ e = q$. We obtain a
factorization $\langle p,q \rangle=\langle m,g \rangle \circ e$, and since
$\langle p,q \rangle$ is injective we conclude that $e$ is bijective and thus
$r$ is single-valued.
\end{proof}
\begin{corollary}\label{cor:dco-discgen}
An indexed preorder $\fifa$ is representable by a DCO if and only if it has a
discrete generic predicate. 
\end{corollary}
\begin{proof}
This follows from Proposition~\ref{prop:dco-iff-discrete} together with
Lemma~\ref{lem:genpred-image}. A direct proof is given
in~\cite[Theorem~2.4]{frey2019characterizing}.
\end{proof}
\begin{remark}
It is possible that the same indexed preorder has discrete and non-discrete
generic predicates: if $\fifa$ is an indexed preorder with discrete generic
predicate $\iota\in\fifa(A)$ and $f:B\epi A$ is a surjection, then $f^*A$ is a
generic predicate which is discrete only if $f$ is a bijection. If $f$ is not a
bijection, we obtain a DCO-representation of $\fifa$ with underlying set $A$,
and a representation as a non-discrete uniform preorder with underlying set $B$.
\end{remark}
\begin{remark}[Cartesian DCOs]\label{rem:cart-dco}If a cartesian uniform preorder $(A,R)$ is a DCO, then the relations
$\lambda,\rho\in R$ from Lemma~\ref{lem:adj-cond} are partial functions, and
jointly form a retraction $\langle\lambda,\rho\rangle:A\pto A\x A$ of $\wedge :
A\x A\to A$, i.e.\ we have $\langle \lambda,\rho \rangle\circ \wedge = \id_{A\x
A}$. Moveover, although we don't have $\wedge\circ \langle \lambda,\rho
\rangle=\id_A$, we have an inclusion $\wedge\circ \langle \lambda,\rho
\rangle\subs \id_A$ of partial functions, since by construction $\lambda$ an
$\rho$ are only defined on the range of $\wedge$.

More generally we define $n$-ary versions
\begin{equation}
\nwedge{n}:A^n\to A\qtext{for}n\in\N
\quad\qqtext{and}\quad
\npi{n}{i}\in R\qtext{for}1\leq i\leq n
\end{equation}
by $\nwedge{0}(\ast)=\top$, $\nwedge{n+1}(\vec a,b)=\nwedge{n}(\vec a)\wedge b$,
and $\npi{n}{i}=\rho\circ \lambda^n_i$, so that we have
\begin{equation}
\langle\npi{n}{1},\dots,\npi{n}{n}\rangle\circ\nwedge{n}=\id_{A^n}
\qquad\text{and}\qquad
\nwedge{n}\circ\langle\npi{n}{1},\dots,\npi{n}{n}\rangle\subs \id_A
\end{equation}
for all $n\in \N$. Loosely following Hofstra~\cite[pg.~254]{hofstra2006all}, we
introduce the notation
\begin{align}
R^{(n)}\;&=\;\setof{r\subs A^n\x A}{\ex s\in R\;.\; r\,=\,s\circ \nwedge{n}}\\
&=\;\setof{r\subs A^n\x A}{r\circ\langle\npi{n}{1},\dots,\npi{n}{n}\rangle\in R}
\end{align}
for `$n$-ary computable' functions, which can be viewed as representing
`multi-inequalities' $\vf_{1},\dots, \vf_{n}\leq\psi$ matching the form of
intuitionistic sequents. A paradigmatic example is given by the DCO of
subrecursive functions (Example~\ref{ex:cuords}\ref{ex:cuords:rec}): here
$R^{(n)}$ contains precisely the \emph{$n$-ary partial sub-recursive functions},
i.e.\ sub-functions of $n$-ary partial recursive functions in the usual sense.
\end{remark}

\section{Partial combinatory algebras}\label{se:pcas}

Partial combinatory algebras can be viewed as trivially ordered OPCAs, but there
is a slight mismatch with the traditional definition of PCA which we
address\,---\,following Streicher~\cite{streicherreal}\,---\,by introducing the
term of \emph{weak} PCA.
\begin{definition}\label{def:pca}
\begin{enumerate}
\item\label{def:pca:wpca}
A \emph{weak partial combinatory algebra} (weak PCA) is a discretely
ordered OPCA, i.e.\ a pair $(A,\ap)$ such that $(A,=,\ap)$ is an OPCA. 
\item\label{def:pca:pca} A \emph{partial combinatory algebra} (PCA) is a weak
PCA in which the element $\coms$ from
Proposition~\ref{prop:opas-wcombicomp}\ref{prop:opas-wcombicomp-ks} can be
chosen such that $\coms\ap a \ap b\ap c\down$ (if and) only if $a\ap c\ap (b\ap
c)\down$.
\end{enumerate}
\end{definition}
There are obvious `filtered' versions of these definitions, for which we use the
adjecrive `relative' as is more common in the unordered case.
\begin{definition}\label{def:rpca}
\begin{enumerate}
\item\label{def:rpca:wrpca}
A \emph{weak relative PCA} is a triple $\aahd$ where $(A,\ap)$ is a
PCA and $A\hs\subs A$ is a filter in the sense of
Definition~\ref{def:opca}\ref{def:opca:filter}.
\item\label{def:rpca:rpca} A \emph{relative PCA} is a weak relative PCA in which
the $\coms\in A\hs$ can be chosen to satisfy the stronger condition of
Definition~\ref{def:pca}\ref{def:pca:pca}.
\end{enumerate}
\end{definition}

\begin{remarks}\label{rem:rpca}
\begin{enumerate}
\item\label{rem:rpca:trad} Relative PCAs are called \emph{elementary inclusions
of PCAs} in~\cite[Sections 2.6.9 and 4.5]{vanoosten2008realizability}
\item\label{rem:rpca:faber}
Faber and van Oosten showed that for every weak PCA $(A,\ap)$ there is a
PCA $(A,\ast)$ such that inducing the same indexed preorder structure on
$\Set(-,A)$ and thus the same uniform preorder structure $A$ (strictly speaking
their result is phrased in terms of \emph{applicative morphisms}, but the
statement about indexed preorders is an easy
consequence)~\cite[Theorem~5.1]{faber2016effective}. Their argument generalizes
easily to relative PCAs.

\end{enumerate}
\end{remarks}
Specializing the constructions from Section~\ref{se:opca}, every relative (weak)
PCA $\aahd$ gives rise to a \ccomp DCO $(A,R_{A\hs})$ with
a basis given by $\setof{(e\ap-):A\pto A}{e\in A\hs}$. Thus, the fiberwise
ordering of the $\ex$-completion $\famf(D(A,R_{A\hs}))$ is given by 
\begin{equation}
(\vf:I\to A)\leq(\psi:I\to A)\qquad\text{iff}\qquad
\ex e\in A\hs\sfa i\in I\sfa a\in\vf(i)\;.\;e\ap a\in\psi(i)
\end{equation}
and we recognize at once that this is the \emph{relative realizability tripos}
over $\aahd$~\cite[Section~2.6.9]{vanoosten2008realizability}.

\medskip

In the following we sketch the argument that \emph{every} \ccomp
DCO arises from a relative PCA this way. To start, given a \ccomp
DCO $(A,R)$ with $@$ (which we call \emph{generic function} in the discrete
case), we define $(-\cdot-):A\x A\pto A$ by $a\ap b = @(a\wedge b)$ and
$A\hs\subs A$ by
\begin{equation}
A\hs\,:=\, \setof{a\in A}{\{(\top,a)\}\in R}\,=\,\setof{a\in A}{\top\leq a \text{ in }\famf\brar(1)}.
\end{equation}
Note that the elements of $A\hs$ correspond to Hofstra's \emph{designated truth
values}~\cite[pg.~244]{hofstra2006all}. If $a,b\in A\hs$ such that $a\ap b =
@(a\wedge b)$ is defined, then $a\ap b\in A\hs$ since $\top\leq a$ and $\top\leq
b$ implies $\top\leq a\wedge b$; and $a\wedge b\leq @(a\wedge b)$, i.e.\ $A\hs$
is closed under application in $A$.
\begin{proposition}\label{prop:rcdcos-are-relpcas}Let $\brar$ be a \ccomp cartesian DCO.
\begin{enumerate}
\item 
For every $n$-ary polynomial $p[x_1,\dots,x_n]$ over the partial applicative
structure $\aahd$ with coefficients in $A\hs$, the partial evaluation function
$\vec a\mapsto p[\vec a]$ is in $R^{(n)}$ (see Remark~\ref{rem:cart-dco}).
\item For all $n\in\N$ and $r\in R^{(n+1)}$ there exists an $e\in A\hs$ such that
for all $a_1,\dots,a_n,b\in A$,
\begin{itemize}
\item $e\ap a_1\ap\dots\ap a_n \down$, and 
\item $r(a_1,\dots,a_n,b)= e\ap a_1\ap\dots\ap a_n\ap b$ whenever
$r(a_1,\dots,a_n,b)\down$.
\end{itemize}
\item $\aahd$ is a weak relative PCA, and the induced \ccomp DCO
$(A,\dfahs)$ is equal to $(A,R)$.
\end{enumerate}
\end{proposition}
\begin{proof}
This is proved in~\cite[Lemma~2.14]{frey2019characterizing} for the non-relative
case, and the generalization to the relative case is straightforward. Hofstra
proved analogous statements for BCOs and filtered OPCAs
in~\cite[Section~6]{hofstra2006all}.
\end{proof}

\begin{theorem}\label{thm:rtr-char}
The following are equivalent for a tripos $\trip$.\begin{enumerate}
\item $\trip$ is equivalent to a relative realizability tripos over a relative
PCA.
\item $\trip$ has enough $\exists$-prime predicates, and $\prim(\trip)$ has
finite meets and a discrete generic predicate.
\end{enumerate}
\end{theorem}
\begin{proof}Assume first that $\trip = \famf(D(A,R_{A\hs}))$ for a relative PCA $\aahd$.
Then Proposition~\ref{prop:dar-primal-ecomp} implies that $\trip$ has enough
$\exists$-primes and $\trip\simeq\famf(A,\dfahs)$. We have established in
Section~\ref{se:opca} that $\famf(A,\dfahs)$ is an indexed meet-semilattice.
and, it has a discrete generic predicate by
Proposition~\ref{prop:dco-iff-discrete}.

Conversely, assume (ii). Then $\prim(\trip)\incl\trip$ is an $\ex$-completion by
Proposition~\ref{prop:ecomp-if-enough-primes}, and $\prim(\trip)$ is
representable by a relative DCO $(A,R)$ by Corollary~\ref{cor:dco-discgen}. The
DCO $(A,R)$ is cartesian since $\prim(\trip)$ has finite meets, and 
\ccomp since its $\ex$-completion is a tripos. Thus, it comes from a weak
relative PCA by Proposition~\ref{prop:rcdcos-are-relpcas}, and from a relative
PCA by Remark~\ref{rem:rpca}\ref{rem:rpca:faber}.
\end{proof}
\begin{remark}
Theorem~\ref{thm:rtr-char} specializes to a characterization of
\emph{non-relative} realizability triposes by adding the condition that $\trip$
is \emph{two-valued}, i.e.\ $\trip(1)\simeq\{\bot<\top\}$. This is equivalent
to $\prim(\trip)(1)\simeq 1$, a property that is called `shallow'
in~\cite{frey2019characterizing}.
\end{remark}

\bibliographystyle{alpha}

\end{document}